\documentclass[reqno,11pt]{amsart}
\usepackage{amsmath,amssymb,amscd,amsthm,latexsym,graphics,enumerate,engord,hyperref,txfonts}
\usepackage[all]{xy}
\usepackage{manyfoot}
\DeclareNewFootnote{R}[roman]

\usepackage{amsthm,amsmath}
\newtheorem{theorem}{Theorem}
\newtheorem{proposition}[theorem]{Proposition}
\newtheorem{lemma}[theorem]{Lemme}

\newtheorem{example}[theorem]{Example}
\newtheorem{corollary}[theorem]{Corollary}
\newtheorem{remark}[theorem]{Remark}

\def\cat{{\rm{cat}\hskip1pt}}

\def\TC{{\rm{TC}\hskip1pt}}
\def\MTC{{\rm{MTC}\hskip1pt}}
\def\HTC{{\rm{HTC}\hskip1pt}}

\def\MTC{{\rm{MTC}\hskip1pt}}

\def\cl{{\rm{cl}\hskip1pt}}

\def\zcl{{\rm{zcl}\hskip1pt}}

\begin{document}
	\title{On the higher rational topological complexity of certain elliptic spaces}

\author{Said Hamoun}
\address{My Ismail University of Mekn\`es, Department of Mathematics, B. P. 11 201 Zitoune, Mekn\`es, Morocco.}
\email{s.hamoun@edu.umi.ac.ma}


\maketitle
	\begin{abstract} Through this paper, we show that $\TC_r(Z)\leq r\cdot\cat(Z)+\chi_{\pi}(Z)$, for any simply-connected elliptic space $Z$ admitting a pure minimal Sullivan model with a differential of constant length. Here $\chi_{\pi}(Z)$ denotes the homotopy characteristic and $r$ is an integer greater or equals than $2$. We also give a lower bound for $\TC_r$ in the framework of coformal spaces and we compute the exact value of $\TC_r$ for certain families of spaces.
\end{abstract}	
\section*{Introduction}
  Let $Z$ be a topological space and fix an integer $r\geq 2$. The higher topological complexity of $Z$, denoted $\TC_r(Z)$, is a homotopy invariant introduced by Rudyak \cite{RY}. Explicitly, it is defined as the least integer $n$ for which $Z^r=Z\times \cdots\times Z$ can be covered by $n+1$ open subset $U_0,\cdots, U_n$ and over each $U_i$ there exists a local continuous section $s_i: U_i\rightarrow Z^r $ of the evaluation map $$ev_r: Z^{[0,1]}\rightarrow Z^r,\quad \gamma\mapsto \left(\gamma(0), \gamma(\frac{1}{r-1}), \gamma(\frac{2}{r-1}),\cdots,\gamma(\frac{r-2}{r-1}), \gamma(1)\right).$$ Whenever $r=2$ this notion coincides with the notion of topological complexity introduced by M. Farber. In this work all topological spaces are considered to be simply-connected CW-complexes of finite type. Note that the computation of the exact value of $\TC_r(Z)$ is often hard, which motivates the study of algebraic approximations. As a well-known lower bound of $\TC_r(Z)$, we consider the higher rational topological complexity $\TC_r(Z_0)$ where $Z_0$ is a rationalization of $Z$. Besides, there exist several other lower bounds such as $\MTC_r(Z)$ and $\HTC_r(Z)$ which are the higher versions of the invariants $\MTC(Z)$ and $\HTC(Z)$ introduced in \cite{FGKV} and \cite{CKV}.  Through rational homotopy theory, several works have been established on $\TC_r(Z_0)$ and we will use in particular a result of Carrasquel which permits us to characterize $\TC_r(Z_0)$ in terms of Sullivan models. Recall that a Sullivan model of $Z$ is a commutative differential graded algebra (cdga for short) $(\Lambda V,d)$ satisfying $H^*(\Lambda V)=H^*(Z;\mathbb{Q})$. Moreover $(\Lambda V,d)$ is said minimal whenever $dV\subset \Lambda^{\geq 2}V$ and in that case we have $V\cong \pi_*(Z)\otimes \mathbb{Q}$. In particular if $dV\subset \Lambda^{l}V$ for such an integer $l\geq 2$, we say that the differential $d$ is of constant length $l$. Since $(\Lambda V,d)$ completely determines the homotopy type of $Z_0$, we will frequently use the notation $\TC_r(\Lambda V)$ instead of $\TC_r(Z_0)$. For this study we focus our interest on pure spaces, that are spaces for which the minimal model $(\Lambda V,d)$ satisfies $dV^{even}=0$ and $dV^{odd}\subset \Lambda V^{even}$. Our goal is the generalization to the higher topological complexity of results established for the classical rational topological complexity $\TC(Z_0)=\TC_2(Z_0)$\footnote{Usually the rational topological complexity  of $Z$ is denoted by $\TC_0(Z)$ but to avoid confusion we here only use the notation $\TC(Z_0$).} in the context of elliptic spaces \cite{HRV}, \cite{HRV1}. A topological space $Z$ is said elliptic if $\pi_*(Z)\otimes \mathbb{Q}$ and $H^*(Z;\mathbb{Q})$ are both finite dimensional. Besides, we will need as well the notion of homotopy characteristic of $Z$ denoted by $\chi_{\pi}(Z):=\dim \pi_{even}(Z)\otimes \mathbb{Q}-\dim \pi_{odd}(Z)\otimes \mathbb{Q}$ which coincides with $\chi_{\pi}(\Lambda V):=\dim V^{even}-\dim V^{odd}$. We first give an explicit expression of the higher rational topological complexity in terms of rational LS-category for formal spaces. 
 Recall that $Z$ is said formal whenever $(\Lambda V,d) \xrightarrow{\simeq} (H(Z;\mathbb{Q}),0)$ and here we have 
\begin{theorem}
	Let $(\Lambda V,d)$ be a pure elliptic minimal model which is formal. Then
	$$\TC_r(\Lambda V)=r \cdot \cat(\Lambda V)+\chi_{\pi}(\Lambda V).$$
\end{theorem}
In the non-formal case the previous equality is not known. But we prove, in a particular case, that the same expression gives an upper bound for the higher rational topological complexity. More precisely:
\begin{theorem}
	Let $(\Lambda V,d)$ be an elliptic pure minimal model together with a differential of constant length. Then 
	$$\TC_r(\Lambda V)\leq r\cdot\cat(\Lambda V)+\chi_{\pi}(\Lambda V).$$ 
\end{theorem}
The last part is devoted to the particular case of coformal spaces that is, when the differential is of constant length $2$. In this context we improve the general lower bound of the higher rational topological complexity given by $(r-1)\cdot \cat(\Lambda V)$ through the following theorem
\begin{theorem}
	Let $(\Lambda V,d)$ a coformal pure elliptic minimal model and $\mathcal{B}$ is a basis of $V^{even}$. Then $$(r-1)\cdot \cat(\Lambda V)+ L(\Lambda V,\mathcal{B})\leq \TC_r(\Lambda V).$$ 
\end{theorem}
Here $L(\Lambda V,\mathcal{B})$ denotes a certain cuplength introduced in \cite{HRV1}. This result is obtained through the following theorem, which improves for the spaces in consideration Lemma $4.1$ of Kishimoto-Yamaguchi \cite{KY}.
\begin{theorem}
For any elliptic pure coformal minimal model $(\Lambda V,d)$, we have
$$ \cat(\Lambda V)+\MTC_r(\Lambda V)\leq \MTC_{r+1}(\Lambda V).$$
\end{theorem}
\section*{Preliminaries}
Let $(\Lambda V,d)$ represent a  minimal Sullivan model of $Z$ and denote by $(\Lambda V)^{\otimes r}:=\Lambda V\otimes \Lambda V\otimes \cdots \otimes \Lambda V$ $r-$times. We consider the projection $$\rho^r_{n}: ((\Lambda V)^{\otimes r},d)\rightarrow \left(\frac{(\Lambda V)^{\otimes r}}{ (\ker \mu^r_{ \Lambda V})^{n+1}} ,\bar{d}\right)$$
where $\mu^r_{ \Lambda V}: (\Lambda V)^{\otimes r}\rightarrow \Lambda V$ is the multiplication of the algebra $\Lambda V$. Recall that $\rho^r_{n}$ admits a homotopy retraction if there exists a cdga morphism $\tau$ fitting in the following diagram
\begin{equation*}
	\xymatrix@=4pc{
		& (\Lambda V)^{\otimes r} \otimes \Lambda W \ar@/_2.5pc/@{-->}[dl]_-{\tau}  \ar[d]^{\xi}_-{\simeq}\\
		(\Lambda V)^{\otimes r} \ar@{^{(}->}[ur]_-{i} \ar[r]_-{\rho^r _ {n}} & \frac{(\Lambda V)^{\otimes r}}{ (\ker \mu^r _{\Lambda V})^{n+1}}
	}
\end{equation*}
In what follows, for any $v\in V$, we use the notation $v(j)= 1\otimes 1\otimes \cdots\otimes v\otimes 1\otimes \cdots \otimes 1\in (\Lambda V)^{\otimes r}$ where $v$ is at the $j^{th}$ position. We also recall that if $(\Lambda V,d)=(\Lambda(v_1, \cdots , v_n),d)$, then $\ker \mu^r_{ \Lambda V}$ is an ideal of $(\Lambda V)^{\otimes r} $ generated by $\{ v_i(j) -v_i(j+1) : i=1, \cdots , n \text{ and } j=1,\cdots ,r-1\}$, see \cite[Lemma 2.3]{JMP}.
In the following, we recall the higher rational topological complexity characterization established by Carrasquel \cite{C2}.
\begin{theorem} Let $Z$ be a simply-connected CW-complex of finite type and $(\Lambda V,d)$ its minimal model. Then $\TC_r(Z_0)=\TC_r(\Lambda V)$ satisfies
\begin{center}
$\TC_r(\Lambda V)\leq m$ if and only if $\rho^r_m$ admits a homotopy retraction.
\end{center}	
\end{theorem}
Moreover it follows from \cite{C2} that the invariants  $\MTC_r(Z)=\MTC_r(\Lambda V)$ and $\HTC_r(Z)=\HTC_r(\Lambda V)$, which satisfy $\HTC_r(\Lambda V)\leq \MTC_r(\Lambda V)\leq \TC_r(\Lambda V),$ can be characterized as follows.
\begin{itemize}
\item  $\MTC_r(\Lambda V)$ is the least integer $n$ for which $\rho^r_n$ admits a homotopy retraction of $(\Lambda V)^{\otimes r}-$modules. 
\item  $\HTC_r(\Lambda V)$ is the least integer $n$ for which $\rho^r_n$ is injective in cohomology. 
\end{itemize}

\section*{The higher rational topological complexity of formal spaces}

 Through the current section we give a specific formula for the higher rational topological complexity of formal spaces. Recall that for this type of spaces it is already known that $\TC_r(\Lambda V)=\zcl_r(\Lambda V)$ and $\cat(\Lambda V)=\cl_{\mathbb{Q}}(\Lambda V)$ where $\zcl_r(\Lambda V)$ and $\cl_{\mathbb{Q}}(\Lambda V)$ are the $r-$zero-divisor-cuplength and the cuplength of $\Lambda V$, respectively. In a first step, we establish an equality in terms of rational LS-category for $F_0$-minimal models. Such spaces are elliptic spaces with a zero homotopy characteristic and are formal. This result can be seen as a reformulation of a result established by Carrasquel (in \cite{C3}).
\begin{lemma} \label{lem6}Let $(\Lambda V,d)$ be an $F_0$-model. Then
	$$ \TC_r(\Lambda V)=r\cdot \cat(\Lambda V) \text{ or equivalently } \zcl_r(\Lambda V)=r \cdot \cl_{\mathbb{Q}}(\Lambda V).$$
\end{lemma}
\begin{proof}
From \cite{FKS} (see also \cite{C3}) we have $\zcl_r(\Lambda V) \geq \zcl_{r-1}(\Lambda V)+\cl_{\mathbb{Q}} (\Lambda V) $. Iteratively, we obtain $\zcl_r(\Lambda V) \geq \zcl(\Lambda V)+(r-2)\cdot \cl_{\mathbb{Q}}(\Lambda V)$. Since an $F_0$-model has its cohomology concentrated in even degrees, it is known by \cite{C3} that $\zcl(\Lambda V)=2\cl_{\mathbb{Q}}(\Lambda V)$. Consequently $\zcl_r(\Lambda V) \geq r\cdot \cl_{\mathbb{Q}}(\Lambda V)$. The converse inequality is immediate and by formality we obtain the equality $\TC_r(\Lambda V)=r\cdot \cat(\Lambda V)$.
\end{proof}

\begin{lemma}\label{lm8}
If $(\Lambda V,d)= (\Lambda V',d)\otimes (\Lambda (z_1,\cdots, z_l),0)$ is a minimal model where $z_1,\cdots ,z_l$ are cocycles of odd degrees. Then $$ \zcl_r(\Lambda V,d)= \zcl_r(\Lambda V',d) +l\cdot(r-1).$$
\end{lemma}
\begin{proof}
Let $\zcl_r(\Lambda V')=p$, then there exist $p$ cocycles $\alpha_1, \cdots, \alpha_p \in \ker\mu^r_{ \Lambda V'}$ satisfying  $[\alpha_1] \cdots [\alpha_p]\neq 0$. As $H(\Lambda V)=H(\Lambda V') \otimes \Lambda(z_1,\cdots, z_l)$ we see that $$ [\alpha_1] \cdots [\alpha_p]\prod_{i=1}^l \prod_{j=1}^{r-1}[z_i(j)-z_i(j+1)]\neq 0.$$
Consequently $ \zcl_r(\Lambda V,d)\geq \zcl_r(\Lambda V',d) +l (r-1)$. Conversely, we have $$\ker \mu^r_{ H(\Lambda V)}= \ker \mu^r_{ H(\Lambda V')}\otimes (\Lambda(z_1,\cdots , z_l))^{\otimes r} + H(\Lambda V')^{\otimes r}\otimes \ker \mu^r_{ \Lambda (z_1,\cdots , z_l)}$$  
where $\mu^r_{ H(\Lambda V)}$ and $\mu^r_{ H(\Lambda V')}$ are respectively the multiplications of $H(\Lambda V)^{\otimes r}$ and $H(\Lambda V')^{\otimes r}$. By computation, and taking into consideration the previous inequality besides the facts $(\ker \mu^r_{ H(\Lambda V')})^{> \zcl_r(\Lambda V')}=0$  and $(\ker \mu^r_{\Lambda (z_1,\cdots ,z_l) })^{> l(r-1)}=0$, we can check that $(\ker \mu^r_{ H(\Lambda V)})^{\zcl_r(\Lambda V')+l (r-1) +1}=0$. That consequently implies  $\zcl_r(\Lambda V)\leq \zcl_r(\Lambda V')+l\cdot(r-1)$. 
\end{proof}
\begin{theorem} \label{th8}
Let $(\Lambda V,d)$ be a pure elliptic minimal model which is formal. Then
$$\TC_r(\Lambda V)=r \cdot \cat(\Lambda V)+\chi_{\pi}(\Lambda V).$$
\end{theorem}
\begin{proof}
According to \cite[Lemma 1.5]{AM}, $(\Lambda V,d)$ is formal if and only if
$$(\Lambda V,d)= (\Lambda V',d)\otimes (\Lambda (z_1,\cdots, z_l),0)$$ 
where $(\Lambda V',d)$ is an $F_0$-model and each $z_i$ is a generator of odd degree. Besides to Lemma \ref{lm8}, we obtain \begin{eqnarray*}
	\TC_r(\Lambda V)&=&\zcl_r(\Lambda V)\\
	&=& \zcl_r( \Lambda V' \otimes (z_1,\cdots, z_l))\\
	&=&  \zcl_r( \Lambda V')+l\cdot(r-1).
\end{eqnarray*} 
As $( \Lambda V',d)$ is an $F_0$-model, by Lemma \ref{lem6} we have
 $\zcl_r( \Lambda V')=r\cdot \cl_{\mathbb{Q}}(\Lambda V')=r\cdot \cat(\Lambda V')$. It is also known that $\cl_{\mathbb{Q}}(\Lambda V')= \cl_{\mathbb{Q}}(\Lambda V')+l$, and we consequently have
\begin{eqnarray*}
	\TC_r(\Lambda V)&=&  \zcl_r( \Lambda V')+l\cdot(r-1)\\
	&=&r \cdot \cat( \Lambda V')+l\cdot(r-1)\\
	&=& r\cdot (\cat( \Lambda V)-l)+l\cdot(r-1)\\
	&=& r\cdot \cat( \Lambda V)-l.
\end{eqnarray*}
On the other hand we have
\begin{eqnarray*}
	\chi_{\pi}(\Lambda V)&=&\dim V^{even}-\dim V^{odd}\\
	&=& \dim (V')^{odd}-(\dim (V')^{odd}+l )\\
	&=& -l.
\end{eqnarray*}
We finally conclude that $\TC_r(\Lambda V) = r\cdot \cat(\Lambda V)+\chi_{\pi}(\Lambda V)$.
\end{proof}

\section*{An upper bound of the higher rational topological complexity of certain pure elliptic spaces}
In this section we will extend to the higher topological complexity the upper bounds we obtained in \cite{HRV} and \cite{HRV1} specifically \cite[Theorem 3.1]{HRV}. The first step is given by the following theorem which is the higher version of \cite[Theorem 3.2]{HRV1}. We note that Minowa has very recently generalized Theorem 3.2 of \cite{HRV1} to the context of parametrized higher rational topological complexity \cite{YM}.
\begin{theorem}
Let $(\Lambda V,d)$ be a minimal model and $u$ an element of odd degree such that $du\in \Lambda V$. Then $(\Lambda V\otimes \Lambda u,d)$ is an extension of $(\Lambda V,d)$ satisfying
\begin{itemize}
	\item[(i)]$\TC_r(\Lambda V \otimes \Lambda u)\leq \TC_r(\Lambda V) +r-1 $. 
	\item[(ii)]$\MTC_r(\Lambda V \otimes \Lambda u)\leq \MTC_r(\Lambda V) +r-1$.
	\item[(iii)]$\HTC_r(\Lambda V \otimes \Lambda u)\leq \HTC_r(\Lambda V) +r-1$.
\end{itemize} 
\end{theorem}
\begin{proof}
Let $\TC_r(\Lambda V)=n$ then there exists  $\tau$ a  retraction of $i$ making commutative the following diagram
\begin{equation*}
	\xymatrix@=4pc{
		& (\Lambda V)^{\otimes r} \otimes \Lambda W \ar@/_2.5pc/@{-->}[dl]_-{\tau}  \ar[d]^{\xi}_-{\simeq}\\
		(\Lambda V)^{\otimes r} \ar@{^{(}->}[ur]_-{i} \ar[r]_-{\rho ^r_ {n}} & \frac{(\Lambda V)^{\otimes r}}{ (\ker \mu^r _{\Lambda V})^{n+1}}
	}
\end{equation*}  
As $\ker \mu^r_{\Lambda u}$ is generated by $\{u(j)-u(j+1): 1\leq j\leq r-1\}$ then $(\ker \mu^r_{\Lambda u})^r=0$. On the other hand we have $$\ker \mu^r_{\Lambda V\otimes \Lambda u}= \ker\mu^r_{ \Lambda V}\otimes (\Lambda u)^{\otimes r} +(\Lambda V)^{\otimes r} \otimes \ker \mu^r_{\Lambda u} $$
we obtain 
\begin{eqnarray*}
	(\ker \mu^r_{\Lambda V\otimes \Lambda u})^{n+r}&=  & \left(\ker\mu^r_{ \Lambda V}\otimes (\Lambda u)^{\otimes r} +(\Lambda V)^{\otimes r} \otimes \ker \mu^r_{\Lambda u}  \right)^{n+r}\\
	&=& (\ker\mu^r_{ \Lambda V}\otimes (\Lambda u)^{\otimes r})^{<n}\cdot ((\Lambda V)^{\otimes r} \otimes \ker \mu^r_{\Lambda u})^{>r}\\
	&+& (\ker\mu^r_{ \Lambda V}\otimes (\Lambda u)^{\otimes r})^{\geq n+1}\cdot ((\Lambda V)^{\otimes r} \otimes \ker \mu^r_{\Lambda u})^{<r} \\
	&=& (\ker\mu^r_{ \Lambda V}\otimes (\Lambda u)^{\otimes r})^{\geq n+1}\cdot ((\Lambda V)^{\otimes r} \otimes \ker \mu^r_{\Lambda u})^{<r}\\
	&\subseteq & (\ker \mu^r_{\Lambda V})^{n+1}
\end{eqnarray*}

and then the composition $$(\ker \mu^r_{\Lambda V\otimes \Lambda u})^{n+r}\hookrightarrow (\Lambda V\otimes \Lambda u)^{\otimes r}\xrightarrow{\rho^r _{n} \otimes (id_{\Lambda u}) ^{\otimes r}} \frac{(\Lambda V \otimes \Lambda u)^{\otimes r}}{(\ker \mu^r_{\Lambda V})^{n+1}}$$
maps to zero. We consequently get a decomposition of $\rho^r_n \otimes (id_{\Lambda u})^{\otimes r}$ as 
\[
\xymatrix{
	(\Lambda V)^{\otimes r} \otimes (\Lambda u)^{\otimes r}\ar[dr]_-{\rho^r_{n+r-1}} \ar[rr]^{\rho ^r_n \otimes (id_{\Lambda u}) ^{\otimes r}} && \frac{(\Lambda V)^{\otimes r}}{(\ker \mu^r _{\Lambda V})^{n+1}} \otimes (\Lambda u)^{\otimes r}\\
	&\frac{(\Lambda V\otimes \Lambda u)^{\otimes r} }{ (\ker \mu ^r_{\Lambda V \otimes \Lambda u})^{n+r}}. \ar[ru] 
}
\]

Let $(\Lambda V\otimes \Lambda u)^{\otimes r}\hookrightarrow (\Lambda V\otimes \Lambda u)^{\otimes r}\otimes \Lambda Z \xrightarrow{\simeq} \frac{(\Lambda V \otimes \Lambda u)^{\otimes r}}{ (\ker \mu^r _{\Lambda V \otimes \Lambda u})^{n+r}}$ be a relative minimal model of $\rho^r_{n+r-1}:(\Lambda V \otimes \Lambda u)^{\otimes r} \rightarrow \frac{(\Lambda V \otimes \Lambda u)^{\otimes r}}{ (\ker \mu^r _{\Lambda V \otimes \Lambda u})^{n+r}}$. In addition, by taking into consideration the diagram:  
\begin{equation*}
	\xymatrix@=4pc{
		(\Lambda V)^{\otimes r} \otimes  (\Lambda u)^{\otimes r}  \ar[r]^{i\otimes (id_{\Lambda u})^{\otimes r}} \ar@{->}[dr]^{\quad\rho_n^r \otimes (id_{\Lambda u})^{\otimes r}} &  (\Lambda V)^{\otimes r} \otimes  \Lambda W \otimes (\Lambda u)^{\otimes r} \ar[d]^{\xi \otimes (id_{\Lambda u})^{\otimes r}}_-{\simeq}\\
		& \frac{(\Lambda V)^{\otimes r}}{ (\ker \mu^r _{\Lambda V})^{n+1}} \otimes (\Lambda u)^{\otimes r} 	}
\end{equation*}  
we obtain the following commutative square
\[
\xymatrixcolsep{0.02pc}\xymatrix{
		(\Lambda V)^{\otimes r}  \otimes (\Lambda u)^{\otimes r}
		\ar@{^{(}->}[d] \ar[rrr]^{i\otimes (id_{\Lambda u})^{\otimes r}}  &&&((\Lambda V)^{\otimes r} \otimes \Lambda W \otimes(\Lambda u)^{\otimes r}) \ar[d]^{\xi \otimes (id_{\Lambda u})^{\otimes r}} _-{\simeq}
		\\
		((\Lambda V \otimes \Lambda u)^{\otimes r}  )\otimes \Lambda Z \ar@{-->}[urrr]^{\tilde{\tau}}  \ar[rr] && \frac{(\Lambda V \otimes \Lambda u )^{\otimes r}}{(\ker \mu ^r_{\Lambda V \otimes \Lambda u})^{n+r}} \ar[r] & \frac{(\Lambda V )^{\otimes r}}{(\ker \mu^r _{\Lambda V})^{n+1}} \otimes (\Lambda u)^{\otimes r}}
		\]
and according to the lifting lemma we ensure the existence of $\tilde{\tau}$ making commutative the previous diagram. Now we can check that the composition $(\tau \otimes (id _{\Lambda u})^{\otimes r} )\circ \tilde{\tau}$ represents a homotopoy retraction for $\rho^r_{n+r-1}:(\Lambda V \otimes \Lambda u)^{\otimes r} \rightarrow \frac{(\Lambda V \otimes \Lambda u)^{\otimes r}}{ (\ker \mu^r _{\Lambda V \otimes \Lambda u})^{n+r}}$. We finally conclude that $\TC_r(\Lambda V\otimes \Lambda u) \leq n+r-1=\TC_r(\Lambda V)+r-1$.\\
For $(i)$ and $(ii)$, we proceed similarly as in $(i)$ by considering respectively a $(\Lambda V)^{\otimes r}$-retraction and a linear differential instead of (cdga) homotopy retraction.

\end{proof}
By applying the previous theorem several times we obtain the following corollary. 
\begin{corollary} \label{cor10}
	Let $(\Lambda V,d)$ be a minimal model and $U$ is a graded vector space concentrated in odd degrees such that $dU \subset \Lambda V$. Then
	\begin{itemize}
		\item[(i)] $\TC_r(\Lambda V \otimes \Lambda U) \leq \TC_r(\Lambda V ) +(r-1)\cdot \dim U$.
		\item[(ii)] $\MTC_r(\Lambda V \otimes \Lambda U) \leq \MTC_r(\Lambda V ) +(r-1)\cdot\dim U.$
		\item[(iii)] $\HTC_r(\Lambda V \otimes \Lambda U) \leq \HTC_r(\Lambda V ) +(r-1)\cdot\dim U$.  
	\end{itemize}
\end{corollary}
We finally establish the following upper bound. 
\begin{theorem}\label{th10}
Let $(\Lambda V,d)$ be a pure elliptic minimal model together with a differential of constant length. Then 
$$\TC_r(\Lambda V)\leq r\cdot \cat(\Lambda V)+\chi_{\pi}(\Lambda V).$$ 
\end{theorem}

\begin{proof}
By a structure theorem \cite[Theorem 1.2]{HRV}  there exists an extension $(\Lambda Z,d)\hookrightarrow (\Lambda V,d)$ where $(\Lambda Z,d) $ is an $F_0$-model satisfying $Z^{even}=V^{even}$. Here, $(\Lambda V,d)$ can be decomposed as $\Lambda V=\Lambda Z\otimes \Lambda U$ where $U$ is a graded vector subspace concentrated in odd degrees and by Corollary \ref{cor10} we have \begin{eqnarray*}
\TC_r(\Lambda V)&=&\TC_r( \Lambda Z\otimes \Lambda U)\\
&\leq& \TC_r( \Lambda Z)+ (r-1) \cdot \dim U.
\end{eqnarray*} 
Moreover $(\Lambda Z,d)$ is an $F_0-$model then $\TC_r(\Lambda Z)=r\cdot \cat(\Lambda Z)$ (see Theorem \ref{th8}). By \cite{LM2} we have $ \cat(\Lambda Z)=\dim Z^{odd}+(l-2)\cdot\dim Z^{even}$ where $l$ is the length of the differential $d$. We then have
\begin{eqnarray*}
	\TC_r(\Lambda V) &\leq&  r\cdot \left(\dim Z^{odd}+(l-2)\cdot \dim Z^{even}\right) +(r-1)\cdot\dim U\\
	&\leq& r\cdot \left(\dim V^{odd}-\dim U+(l-2)\cdot\dim V^{even}\right) +(r-1)\cdot \dim U\\
	&\leq& r\cdot \left(\dim V^{odd}+(l-2)\cdot \dim V^{even}\right) -r\cdot \dim U +(r-1)\cdot\dim U\\
	&\leq &r\cdot \cat(\Lambda V) -\dim U.
\end{eqnarray*} 
On the other hand we have 
\begin{eqnarray*}
\chi_{\pi}(\Lambda V)&=&\dim V^{even}-\dim V^{odd}\\
&=& \dim Z^{odd}-(\dim Z^{odd}+\dim U )\\
&=& -\dim U.
\end{eqnarray*}
Consequently we conclude that $\TC_r(\Lambda V) \leq r\cdot\cat(\Lambda V)+\chi_{\pi}(\Lambda V)$.
\end{proof} 
\begin{example}
	Let consider the homogeneous space $\frac{SU(6)}{SU(3)\times SU(3)}$ whose minimal model is given by $(\Lambda V,d)=\Lambda(x_1,x_2,y_1,y_2,z)$ where $|x_1|=4,$ $|x_2|=6,$ $dx_1=0,$ $dx_2=0,$ $dy_1=x_1^2,$ $dy_2=x_2^2$ and $dz=x_1x_2$. For this example we have $\cat(\Lambda V)=\dim V^{odd}=3$ since $d$ is quadratic and $\chi_{\pi}(\Lambda V)=-1$ which implies $\TC_r(\Lambda V)\leq 3r-1.$
\end{example}
\section*{The coformal case}
In this section we study the rational higher topological complexity of pure elliptic coformal spaces. We first recall the length $L(\Lambda V, \mathcal{B})$ associated to a (non-necessarily elliptic) pure coformal model $(\Lambda V,d)$ given with a basis $\mathcal{B}=\{x_1,\cdots,x_n\}$ of $X:=V^{even}$. We extend $(\Lambda V,d)$ to a coformal pure elliptic model $(\Lambda W _\mathcal{B}=\Lambda (V\oplus U),d)$ by adding generators $u_i$ of odd degrees satisfying $du_i=x_i^2$ for each $i=1,\cdots,n$. Since the differential $d$ is quadratic, $d$ can splits into bigraded operators $$d_{p,q}: \Lambda^p X\otimes \Lambda^q (Y\oplus U) \rightarrow \Lambda ^{p+2}X\otimes \Lambda^{q-1} (Y\oplus U)$$ 
which induce a bigraded cohomology $H_{p,q}(\Lambda W_{\mathcal{B}})$. In \cite{HRV1} we have introduced the following length $$ L(\Lambda V, \mathcal{B})=\max\{i: [z_1]\cdots [z_i]\neq 0,\quad [z_1],\cdots, [z_i]\in H_{odd,*}(\Lambda W_{\mathcal{B}}) \}.$$
This notion permits us to improve the general lower bound of the higher rational topological complexity given by $(r-1)\cdot \cat(\Lambda V)$.
\begin{theorem} \label{theo.13}
	Let $(\Lambda V,d)$ be a pure coformal elliptic minimal model. Then $$(r-1)\cdot \cat(\Lambda V)+ L(\Lambda V,\mathcal{B})\leq \TC_r(\Lambda V).$$ 
\end{theorem}
Before to prove this result, let first construct of a cocycle $\Omega$ which will be used later. Let as before consider a basis $\{x_1,\cdots,x_n\}$ of $X=V^{even}$ and let $\{y_1,\cdots,y_m\}$ be a basis of $Y:=V^{odd}$. Consider $\Lambda V \otimes \Lambda \bar{X}\otimes \Lambda V' \otimes \Lambda \bar{X}'$ where $\Lambda V'$ is a copy of $\Lambda V$ and $\bar{X}=sX$ is the
suspension of $X$ with $d$ extended by $d\bar{x}=dx, \forall x\in X $. We adopt $v$ and $v'$, respectively, to reefer to $v\otimes 1$ and $1\otimes v$ for all $v\in \Lambda V$. As the differential $d$ is quadratic, then for any $j=1,\cdots,m$ there exist $\alpha ^j _k , \beta ^j_{p,q} \in \mathbb{Q}$ satisfying $dy_j= \sum _k  \alpha ^j _k x^2_k+ \sum _{p<q} \beta ^j _{p,q}x_{p} \cdot x_{q}$. In $\Lambda V \otimes \Lambda \bar{X}\otimes \Lambda V' \otimes \Lambda \bar{X}'$ we have
\begin{eqnarray*}
	d(y_j-y_j')&=& \sum _k \alpha ^j _k(x^2_{k}- {x'_{k}}^2) + \sum _{p< q} \beta ^j _{p,q}(x_{p} x_{q} -x'_{p} x'_{q}) \\
	&=&d \left( \sum _k \alpha ^j _k (x_{k}- x'_{k})(\bar{ x }_{k} + \bar{ x}'_{k} ) +\frac{1}{2} \sum _{p< q}  \beta ^j _{p,q}(x_{p}- x '_{p})(\bar{ x }_{q} + \bar{ x}'_{q} )\right.\\
	&+& \left. \frac{1}{2} \sum _{p< q}  \beta ^j _{p,q}(x_{q}- x '_{q})(\bar{ x }_{p} + \bar{x}'_{p} )\right), \quad \text{for all} \quad j=1,\cdots ,m.
\end{eqnarray*}
By setting
$$\phi _j=  \sum _k \alpha ^j _k (x_{k}- x'_{k})(\bar{ x }_{k} + \bar{ x}'_{k} ) +\frac{1}{2} \sum _{p< q}  \beta ^j _{p,q}(x_{p}- x'_{p})(\bar{ x }_{q} + \bar{ x}'_{q} )+\frac{1}{2} \sum _{p< q}  \beta ^j _{p,q}(x_{q}- x'_{q})(\bar{ x }_{p} + \bar{ x}'_{p} )$$
we obtain that, the elements
\begin{itemize}
	\item[•] $y_j -y'_j- \phi _j$, with $j=1, \cdots , m$
\end{itemize}
are all cocycles in $\Lambda V   \otimes \Lambda \bar{X}\otimes \Lambda V '\otimes \Lambda \bar{X}'$.\\ 
The cocycle $\Omega$ is defined to be the coefficient of $ \prod \limits _{i=1}^n( \bar{x}_i+\bar{x}_i')$ gotten after developing 
$$\prod _{j=1}^m \left(y_j-y'_j -\phi _j\right).$$
Note that $\Omega\in (\ker \mu _{\Lambda V})^m$. Now let consider the following projection $$ pr: \Lambda V \otimes \Lambda \bar{X} \otimes \Lambda V' \otimes\Lambda \bar{X}'\rightarrow \Lambda V' \otimes\Lambda \bar{X}'$$ we can check that
$$pr(\prod \limits _{i=1}^n( \bar{x}_i+\bar{x}_i')) =\prod \limits _{i=1}^n\bar{x}_i', \text{ and } pr(\prod _{j=1}^m \left(y_j-y'_j -\phi _j\right))=(-1)^m\prod _{j=1}^m \left(y'_j-\varphi _j'\right)$$
where $\varphi_j'=pr(-\phi_j)= \sum _k  \alpha ^j _k x'_k\cdot \bar{ x}'_k+\frac{1}{2} \sum _{p<q} \beta ^j _{p,q} {x}'_p\cdot \bar{x}'_{q} +\frac{1}{2} \sum _{p<q} \beta ^j _{p,q}{x}'_{q}\cdot \bar{x}'_p$. Consequently, the coefficient of $ \prod \limits _{i=1}^n( \bar{x}_i+\bar{x}_i')$ gotten after developing $\prod _{j=1}^m \left(y_j-y'_j -\phi _j\right)$ is projected through the map $pr$ to the coefficient of $ \prod \limits _{i=1}^n\bar{x}_i'$ gotten after developing 
$(-1)^m\cdot \prod _{j=1}^m (y'_j -\varphi _j').$ That is $pr(\Omega)=(-1)^m\omega'$ where $\omega'$ is a cocycle in $\Lambda^m V'$ whose construction corresponds to the  Lechuga-Murillo construction of a cocycle representing the top class of $\Lambda V'$ (see \cite{LM2}). In particular the cohomology class of $\omega'$ is non-trivial. As a second step, let consider the decomposition $$\Lambda^{\leq m}( V\oplus V')=(\Lambda^{m} V\otimes 1) \oplus ( \bigoplus_{h,l}\Lambda^{h} V\otimes \Lambda^{l} V') \oplus (1\otimes \Lambda^{m} V')$$ 
where $h, l \in \mathbb{N}$ satisfying $h+l\leq m$. Then we can decompose $\Omega$ as $$\Omega=\delta+(-1)^m\omega'$$ 
with $\delta \in (\Lambda^{m} V\otimes 1) \oplus (\bigoplus_{h,l}\Lambda^{h} V\otimes \Lambda^{l} V')$ and $\omega' \in 1\otimes \Lambda^{m} V'$. The construction of $\omega'$ and $\Omega$ in combination with a technique similar to that of \cite[Lemma 4.1]{KY} permits us to establish the following result. Note that since $(\Lambda V,d)$ is coformal and elliptic $\cat(\Lambda V)=\dim V^{odd}=m$ (\cite{FH}). 
\begin{theorem}
	For $(\Lambda V,d)$ pure coformal elliptic, we have
	$$\cat(\Lambda V)+ \MTC_r(\Lambda V)\leq \MTC_{r+1}(\Lambda V) .$$
\end{theorem}
\begin{proof}
	Let $\omega$ be the cocycle constructed above. We decompose $\Lambda V$ as $$ \Lambda V=\mathbb{Q}\cdot \omega\oplus Z$$ 
	and consider $I$ the following map $$ I: \Lambda V \otimes \Lambda V \rightarrow \Lambda V$$
	defined by $I(1\otimes \omega )= 1$ and $I( \Lambda V \otimes Z)=0$. Since the cohomology class of $\omega$ is non-trivial, $I$ is a morphism of differential $\Lambda V-$modules. From the construction of $\Omega \in (ker\mu _{\Lambda V})^{m}$, we can chek that $I(\Omega)=(-1)^m$, and without loss of generality, we can modify the definition of $I$ to obtain $I(\Omega )= 1$ and $I( \Lambda V\otimes Z)=0$. We subsequently define a morphism of $(\Lambda V)^{\otimes r}$-modules as  
	$$ I_r:(\Lambda V)^{\otimes r+1} = (\Lambda V)^{\otimes r-1}\otimes \Lambda V\otimes \Lambda V \xrightarrow{id\otimes I} (\Lambda V)^{\otimes r-1} \otimes \Lambda V $$
	which commutes with the differential $d$. On the other hand, as $\Omega$ is a cocycle, we can define the following morphism of $(\Lambda V)^{\otimes r}-$modules 
	$$ f: (\Lambda V)^{\otimes r}\rightarrow (\Lambda V )^{\otimes r+1}$$
	given by $f(x)=x\cdot \Omega$ (Here we identified $\Omega= 1_{ (\Lambda V)^{\otimes r-1}} \otimes \Omega\in (\Lambda V )^{\otimes r+1}$ and $x=x\otimes 1 \in (\Lambda V )^{\otimes r+1}$ for all $x \in (\Lambda V)^{\otimes r}$ ). Note that the morphism  $I_r$ represents a retraction of $(\Lambda V)^{\otimes r}$-modules for $f$. We know that $\MTC_{r+1}(\Lambda V)\geq cat(\Lambda V)=m$. Consequently there exists an integer $p$ satisfying
	$\MTC_{r+1}(\Lambda V)=p+cat(\Lambda V)=p+m$. As $\Omega \in \ker (\mu _{\Lambda V})^{m}\subset \ker (\mu^r _{\Lambda V})^{m} \subset \ker (\mu^{r+1} _{\Lambda V})^{m}$ then any $x \in \ker (\mu^r _{\Lambda V})^{p+1}$ maps through $f$ to $f(x)=x\cdot\Omega \in \ker (\mu^{r+1} _{\Lambda V})^{m+p+1}$ and consequently we obtain the following commutative diagram of differential $(\Lambda V)^{\otimes r}-$modules

	\begin{equation*}
		\xymatrix@=4pc{
			(\Lambda V)^{\otimes r} \ar[d]_-{\rho^{r}_{p}}\ar[r]^f & (\Lambda V)^{\otimes r+1} \ar[d] ^{\rho^{r+1}_{m+p}}\\
			\frac{(\Lambda V)^{\otimes r}}{ (\ker \mu^r _{\Lambda V})^{p+1}} \ar[r]^{\tilde{f}} & \frac{(\Lambda V)^{\otimes r+1}}{ (\ker \mu^{r+1} _{\Lambda V})^{m+p+1}}
		}
	\end{equation*}  
	Since $\MTC_{r+1}(\Lambda V)=m+p$ then $\rho^{r+1}_{m+p}$ admits $\tau$ as a homotopy retraction of differential $(\Lambda V)^{\otimes r}-$modules and then $I_r\circ\tau\circ \tilde{f}$ is a homotopy retraction of differential $(\Lambda V)^{\otimes r}-$modules for $\rho_p^r$. This means that $\MTC_{r}(\Lambda V)\leq p= \MTC_{r+1}(\Lambda V)-\cat(\Lambda V)$ and consequently $\MTC_{r}(\Lambda V)+\cat(\Lambda V)\leq  \MTC_{r+1}(\Lambda V).$
\end{proof}
By an induction process, we obtain the following corollary: 
\begin{corollary}\label{cor15} If $(\Lambda V,d)$ is pure coformal and elliptic, then
	$$ (r-2)\cdot \cat(\Lambda V)+\MTC(\Lambda V)\leq \MTC_r(\Lambda V)$$
	where $\MTC(V)$ is the usual modular rational topological complexity of $\Lambda V$.
\end{corollary}
As an application of the previous theorem and combining with the third axiom of \cite[Lemma 1]{FKS}, we obtain
\begin{corollary}For $(\Lambda V,d)$ pure coformal elliptic, the generating function $$\hat{\mathcal{F}}_{\Lambda V}(x)= \sum _{r=1}^{\infty}\MTC_{r+1}(\Lambda V)\cdot x^r$$
is a rational function of the form $\frac{P(x)}{(1-x)^2}$ where $P(x)$ is an integer polynomial satisfying $P(1)=\cat(\Lambda V).$
\end{corollary}

Now we are ready to prove the result concerning the lower bound of the higher rational topological complexity.
\begin{proof}[Proof of Theorem \ref*{theo.13}]
	Recall, from the begining of this section, the construction of $\Lambda W_{\mathcal{B}}$ and note that this model is also a pure coformal elliptic model. In the proof of theorem \cite[Theorem 5.2]{HRV1}, we have established, for any basis $\mathcal{B}$ of $V^{even}$, that $$cat(\Lambda W_\mathcal{B})+L(\Lambda V, \mathcal{B})\leq \HTC(\Lambda W_\mathcal{B})\leq \MTC(\Lambda W_\mathcal{B}).$$ 
Together with Corollary \ref{cor15}, this implies 
	$$(r-1)\cdot cat(\Lambda W_\mathcal{B})+L(\Lambda V, \mathcal{B})\leq \MTC_r(\Lambda W_\mathcal{B}).$$
	Besides the facts that  $\MTC_r(\Lambda W_\mathcal{B})\leq \MTC_r(\Lambda V)+n\cdot(r-1)$ (see Corollary \ref{cor10}) and $cat(\Lambda W_\mathcal{B})=n+m=cat(\Lambda V)+n$, we obtain
	$$(r-1)\cdot (n+m)+L(\Lambda V, \mathcal{B})\leq \MTC_r(\Lambda W_\mathcal{B})\leq \MTC_r(\Lambda V)+  (r-1)\cdot n,$$
	which consequently implies 
	$$(r-1)\cdot cat(\Lambda V)+L(\Lambda V, \mathcal{B})\leq \MTC_r(\Lambda V)\leq \TC_r(\Lambda V).$$
\end{proof}
\begin{example}
	The homogeneous space $\frac{SU(6)}{SU(3)\times SU(3)}$ has a minimal model of the form $(\Lambda V,d)=(\Lambda(x_1,x_2,y_1,y_2,z ) ,d)$ where $|x_1|=4,$ $|x_2|=6,$ $dx_1=dx_2=0, dy_1=x_1^2, dy_2=x_2^2$ and $dz=x_1x_2$. In \cite{HRV1} we have proved that for this example $L(\Lambda V, \mathcal{B})=2$ for certain basis $\mathcal{B}$ of $V^{even}$. That implies $ 3(r-1)+2=3r-1\leq \TC(\Lambda V)$ (from Theorem \ref{theo.13}) besides to Theorem \ref{th10}, we obtain $\TC_r(\Lambda V)=3r-1$.
\end{example}
\begin{remark} \label{rk16}
	Let $(\Lambda V,d)=(\Lambda (x_1,\cdots,x_n,y_1,\cdots y_m),d)$ be a coformal pure elliptic minimal model. If $n=2$ then we can check that $L(\Lambda V,\mathcal{B})=2$ for any $\mathcal{B}=\{x_1,x_2\}$ a basis of $V^{even}$. In this case we have  $\TC_r(\Lambda V)=r\cdot\cat(\Lambda V)+\chi_{\pi}(\Lambda V)=(r-1)\cdot m+2.$
\end{remark}
\subsection*{Special family}
In this paragraph we are interested in the computation of the higher rational topological complexity of spaces for which the minimal model has the form $(\Lambda W,d)=(\Lambda(x_1,\cdots, x_n,u_1,\cdots ,u_n,y),d)$ where $dx_i=0,$ $du_i=x_i^2$ and $dy\in \Lambda^2(x_1,\cdots, x_n).$
\begin{proposition} \label{pr19}
	For $(\Lambda W,d)=(\Lambda(x_1,\cdots, x_n,u_1,\cdots ,u_n,y),d)$ satisfying $dx_i=0,$ $du_i=x_i^2$ and $dy\in \Lambda^2(x_1,\cdots, x_n)$, we have 
	$$\TC_r(\Lambda W)=r\cdot \cat(\Lambda W)+\chi_{\pi}(\Lambda W)=r\cdot (n+1)-1.$$
\end{proposition}
\begin{proof}
From Corollary \ref{cor15}, we have $(r-2)\cdot \cat(\Lambda W)+ \MTC(\Lambda W) \leq \MTC_r(\Lambda W)$. Besides \cite[Proposition 5.11]{HRV1}, ensuring that $\HTC(\Lambda W)=\MTC(\Lambda W)=\TC(\Lambda W)=2\cat(\Lambda W) +\chi_{\pi}(\Lambda W)$. We obtain
\begin{eqnarray*}
	\MTC_r(\Lambda W) &\geq & (r-2)\cdot \cat(\Lambda W)+2\cat(\Lambda W) +\chi_{\pi}(\Lambda W)\\
	&\geq & r \cdot\cat(\Lambda W) +\chi_{\pi}(\Lambda W).
\end{eqnarray*}
The remaining inequality follows from Theorem \ref{th10}.
\end{proof}
\begin{remark}
From the previous Proposition, we can check that $\TC_{r+1}(\Lambda W)=\TC_{r}(\Lambda W)+\cat(\Lambda W)$ (which is also always true for models in Remark \ref{rk16}). Consequently, the generating function $\mathcal{F}_{\Lambda W}(x)= \sum _{r=1}^{\infty}\TC_{r+1}(\Lambda W)\cdot x^r$ is a rational function of the form $P(x)\cdot(1-x)^{-2}$ satisfying $P(1)=\cat(\Lambda W)$ for models in consideration. That represents additional families supporting a positive answer to the question of Farber and Oprea \cite{FO}.
\end{remark}

\section*{Acknowledgements}
The author would like to express a deep gratitude to Professor Lucile Vandembroucq for the valuable guidance. Sincere thanks are also extended to Professor Youssef Rami for his support.

\end{document}